\documentclass[12pt, reqno,fancysections,pagenumber]{amsart}
\usepackage[T1]{fontenc}
\usepackage[utf8]{inputenc}
\usepackage{lmodern}
\usepackage[english]{babel}
\usepackage{url,csquotes}
\usepackage[hidelinks,hyperfootnotes=false, linkcolor=blue, citecolor=green, colorlinks=true]{hyperref}
\usepackage{geometry}
\usepackage{fullpage}
\usepackage{graphicx}
\usepackage{xcolor}
\usepackage{amsmath, amssymb}
\usepackage{titlesec, titletoc}
\usepackage{xypic}
\usepackage[normalem]{ulem}

\title{The volumes of Miyauchi subgroups}
\author{Didier \textsc{Lesesvre} and Ian \textsc{Petrow}}
\date{\today}

\usepackage{fancyhdr}

\makeatletter
\let\@oddfoot\@empty
\let\@evenfoot\@empty
\makeatother

\makeatletter
\def\@maketitle{%
  \newpage
  \renewcommand{\headrulewidth}{0pt}
  \thispagestyle{empty}
  \null
  \vskip 2em%
  \begin{center}%
    {\Large\sffamily\bfseries \MakeUppercase\@title  \par}%
    \vskip .5em%
  \includegraphics{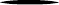}\\
  \vskip 1em
    {\large\mdseries\hspace{0cm}
      \begin{tabular}[t]{c}%
        \sffamily\hspace{0cm} Didier Lesesvre and Ian Petrow
      \end{tabular}\par}%
    \vskip 1em%
  \end{center}%
}
\makeatother

\renewcommand{\headrulewidth}{1pt}

\titleformat{\section}
        [block]
        {\Large\bfseries\sffamily\filcenter}
        {\thesection. }
        {0em}
        {}
        [%\vspace*{-0.5\baselineskip}%
        ]

 \titleformat{\subsection}
        [block]
        {\Large\scshape}
        {\thesubsection}
        {0.5em}
        {}
        [%\vspace*{-0.3\baselineskip}%
        ]
\titleformat{\subsubsection}
    [block]
    {\large\scshape}
    {\thesubsubsection}
    {0.5em}
    {}
    []

\titlecontents{section}[4pc]
{}
{\makebox[0pt][r]{\makebox[30pt][r]{\textsc{\thecontentslabel.\quad}}}}
{}
{\hfill \thecontentspage}

\renewenvironment{abstract}{\bigskip \begin{center}\begin{minipage}[c]{0.8\textwidth} \normalsize}{\end{minipage} \end{center}}

\newcommand{\Z}{\mathbf{Z}}

\newcommand{\p}{\mathfrak{P}}
\newcommand{\fp}{\mathfrak{P}}

\newcommand{\GL}{\mathrm{GL}}
\newcommand{\SL}{\mathrm{SL}}

\newcommand{\fO}{\mathfrak{O}}
\renewcommand{\o}{\mathfrak{o}}

\newcommand{\vol}{\mathrm{vol}}

\newtheorem{lem}{Lemma}
\newtheorem{thm}{Theorem}

\newtheorem{cor}{Corollary}

\usepackage{chngcntr}
\numberwithin{equation}{section}

\usepackage{mathptmx}

\address{\newline School of mathematics (Zhuhai) \newline
Zhuhai Campus, Sun Yat-Sen University \newline
Tangjiawan, Zhuhai, Guangdong, 519082, China (PRC) \newline \textnormal{\texttt{lesesvre@math.cnrs.fr}} \newline \newline University College London \newline Department of Mathematics \newline 25 Gordon Street \newline London WC1H 0AY, United Kingdom
\newline \textnormal{\texttt{i.petrow@ucl.ac.uk}} }

\begin{document}

\maketitle
\large

\begin{abstract}
Miyauchi described the $L$ and $\varepsilon$-factors attached to generic representations of the unramified unitary group of rank three in terms of local newforms defined by a sequence of subgroups. We calculate the volumes of these Miyauchi groups.
\end{abstract}

\setcounter{tocdepth}{1}
\tableofcontents

\section{Introduction}

\subsection{Miyauchi's theory of local newforms}

Let $E/F$ be a quadratic unramified extension of non-archimedean local fields of characteristic zero. We use $\overline{x}$ to denote the nontrivial Galois automorphism applied to $x \in E$. Let $\o$, $\fO$ be the rings of integers of the two fields $F$, $E$, respectively. Let $\mathfrak{p}$, $\p$ be their maximal ideals, and $q_F$, $q_E$ be the cardinalities of their residue fields. Let 
\begin{equation} U= \left\{ g \in \GL_3(E) : {}^t\overline{g} J g = J \right\},\end{equation}
where \begin{equation}\label{Jdef}J = \left( \begin{array}{ccc} & & 1 \\ & 1 & \\ 1 & & \end{array}\right).\end{equation}

The group $U$ is a rank $3$ unitary group defined over $F$. There is a unique isomorphism class of genuine unitary groups over $F$ \cite[Section 1.9]{Rogawski}, so we are content to work with the above explicit model for this single isomorphism class of unitary groups.

Consider the family of compact open subgroups of $U$ defined by for all $n\geqslant 0$ by \begin{equation}\Gamma_n =\left\{   g \in  \left(\begin{array}{ccc} \fO & \fO & \p^{-n} \\ \p^n & 1+ \p^n & \fO \\ \p^n & \p^n & \fO \end{array} \right) : \det (g ) \in \fO^\times\right\} \cap U.\end{equation}
We call these the Miyauchi subgroups, as defined in \cite{Miyauchi3}. These are of particular interest because of the theory of local newforms established by Miyauchi \cite{MiyauchiConductors2}, relating the notions of conductors and newforms to $L$ and $\varepsilon$-factors attached to representations of $U$, as summarized in the theorem below.

Chose a nontrivial unramified additive character $\psi_E$ of $E$ and a non-trivial additive character $\psi_F$ of $F$.  For an irreducible admissible generic representation $\pi$ of $U$, let $\mathcal{W}(\pi, \psi_E)$ be its Whittaker model with respect to $\psi_{E}$.  Gelbart, Piatetski-Shapiro \cite{gelbart_automorphic_1983} and Baruch \cite{Baruch} defined the $L$-factor $L(s, \pi)$ associated to an irreducible generic representation $\pi$ of $U$ as the greatest common divisor of a family of explicit zeta integrals $Z(s, W, \Phi)$ where $W \in \mathcal{W}(\pi, \psi_E)$ and $\Phi$ is a function in $C_c^\infty(F^2)$. The $\varepsilon$-factor $\varepsilon(s, \pi, \psi_E, \psi_F)$ is then defined as the quantity appearing in the functional equation satisfied by $L(s, \pi)$. These fundamental quantities in the theory of automorphic forms on $U$ can be characterized in a more explicit and computational way with the sequence of Miyauchi subgroups.

\lfoot{}

\begin{thm}[Miyauchi]
\label{thm:miyauchi-newforms}
Assume the residue characteristic of $F$ is odd. Let $\pi$ be an irreducible generic representation of $U$. For any $n \geqslant 0$, let $V(n)$ be the subspace of $\pi$ fixed by $\Gamma_n$. Then
\begin{itemize}
\item[(i)] there exists $n \geqslant 0$ such that $V(n)$ is nonzero;
\item[(ii)] if $n_\pi$ is the least such integer, then $V(n_\pi)$ is one-dimensional and its elements are called the newforms for $\pi$;
\item[(iii)] if $v$ is the newform for $\pi$ such that the corresponding Whittaker function $W_v$ satisfies $W_v(I_3) = 1$, then the corresponding Gelbart-Piatetski-Shapiro zeta integral $Z(s, W_v, \mathbf{1}_{\mathfrak{p}^{n_\pi} \oplus \mathfrak{o}})$ matches the $L$-factor $L(s, \pi)$;
\item[(iv)] if $\psi_F$ has conductor $\o$, then the $\varepsilon$-factor of $\pi$ satisfies \begin{equation}\varepsilon(s, \pi, \psi_E,\psi_F) = q_E^{- n_\pi  (s-1/2)}.\end{equation}
\end{itemize}
\end{thm}
\begin{proof} 
Points (i) and (ii) are \cite[Theorem 0.3]{Miyauchi3}, and points (iii) and (iv) are \cite[Theorem 3.6]{Miyauchi4}. 
\end{proof}

Consider the Haar measure on $U$ normalized so that the maximal open compact subgroup $\Gamma_0$ gets volume one. In this short note, we compute the volumes of the Miyauchi subgroups with respect to this measure.  

\begin{thm}
\label{theorem} 
For every $n \geqslant 1$, the volumes of the Miyauchi subgroups are given by
\begin{equation}
{\mathrm{vol}}(\Gamma_n) = q_F^{-3n}(1+ q_F^{-3})^{-1}  .
\end{equation}
\end{thm}
Since the subgroups $\Gamma_n$ characterize the conductor by the point (iv) in Miyauchi's theorem, their volumes should appear as invariants in many problems involving automorphic forms on unitary groups of rank 3.    
Theorem \ref{theorem} is the analogue of e.g.\ \cite[Theorem 4.2.5]{miyake_modular_1989} in the case of $\mathrm{GL}_{2}$, and of \cite[Lemma 3.3.3]{RobertsSchmidt} in the case of $\mathrm{GSp}_{4}$. 

\subsection{Classical theory of local newforms}

Let us turn back to a more usual setting, the one of general linear groups over $F$, where the established importance of local newforms is a suitable motivation to our problem.  For instance, the indices of the classical congruence subgroups in $\SL_2(\Z)$ are known to be equal to
\begin{equation}\label{indices}
[\SL_2(\Z):\Gamma_0(N)] = N \prod_{p \mid N} \left(1+\frac{1}{p}\right),
\end{equation}

\noindent which is a central fact in the theory of automorphic forms, appealed to for many purposes: establishing the dimension of the space of automorphic cusp forms of given level \cite{miyake_modular_1989}, sieving out the oldforms and the spectral multiplicities in arithmetic statistic questions \cite{iwaniec_low_2000}, or motivating normalization in spectral theory and trace formulas \cite{knightly_relative_2006}, for instance. 

In a more representation theoretic framework, the analogous result to Theorem \ref{thm:miyauchi-newforms} was established by Casselman \cite{casselman_results_1973} and Deligne \cite{DeligneNewforms} in the case of $\GL_{2}$ and by Jacquet, Piatetski-Shapiro and Shalika \cite{jacquet_conducteur_1981} in the case of $\GL_{m}$. More precisely, let
\begin{equation}
{\widetilde{\Gamma}_n} = \left(
\begin{array}{cc}
\GL_{m-1}(\mathfrak{o}) & M_{m-1, 1}(\mathfrak{o}) \\
{p}^n M_{1, m-1}(\mathfrak{o}) & 1 +{p}^n
\end{array}
 \right),
\end{equation}

\noindent be the classical congruence subgroups. These are of particular interest because of the following result, relating the theory of classical $L$-functions to the fixed vectors by these congruence subgroups, analogously to Theorem \ref{thm:miyauchi-newforms}. Let $\psi$ be a non-trivial unramified additive character of $F$ viewed as a character of the upper-triangular Borel subgroup of $\GL_m(F)$ as in \cite[Eq.\ (2)]{jacquet_conducteur_1981}. For an irreducible admissible generic representation $\pi$ of $\GL_m(F)$, let $\mathcal{W}(\pi, \psi)$ be its Whittaker model  with respect to $\psi$.

\begin{thm}[Casselman, Jacquet, Piatetski-Shapiro, Shalika]
\label{casselman}
Let $\pi$ be an irreducible generic representation of $\GL_{m}$. {For $n \geqslant 0$, let $V(n)$ be the subspace of $\pi$ fixed by $\widetilde{\Gamma}_n$.} Then
\begin{itemize}
\item[(i)] there exists $n \geqslant 0$ such that {$V(n)$ is nonzero};
\item[(ii)] if {$n_\pi$} is the least such integer, then {$V(n_\pi)$} is one-dimensional and its elements are called the newforms for $\pi$;
\item[(iii)] the $\varepsilon$-factor $\varepsilon(s, \pi, \psi)$ is $q_F^{-{n_\pi}(s-1/2)}$ times a quantity independent of $s$. 
\end{itemize}
\end{thm}

Since these congruence subgroups characterize the conductor by the point (iii) of the above theorem, their volumes govern the sizes of the families of automorphic forms on $\GL_{m}$. The indices \eqref{indices} are therefore necessary to compute the main term in the automorphic Weyl Law in the analytic conductor aspect for $\GL_{2}$ and its inner forms \cite{BrumleyMilicevic, Lesesvre}, where the characterization of the conductor in terms of congruence subgroups is used in a critical way. The strong similarity between Theorem \ref{thm:miyauchi-newforms} and Theorem \ref{casselman},  as well as the fundamental role played by the volumes of the subgroups $\tilde{\Gamma}_n$ in the $\GL_m$ case, suggest that the volumes of the Miyauchi subgroups $\Gamma_n$ will play a central role in the analytic theory of $U$. In particular, we expect that these volumes will appear in the main term of the automorphic Weyl Law for rank 3 unitary groups, see \cite[Chapter 5]{lesesvre_arithmetic_2018}. 

\subsection{Outline of the proof}

We follow the strategy of Roberts and Schmidt \cite[Lemma 3.3.3]{RobertsSchmidt}, but encounter some new difficulties owing to the fact that $U$ is non-split. The proof has several steps.
\begin{enumerate}
\item Consider the subgroups of $\Gamma_n$ given by, for any $n \geqslant 0$, $$A_n = \left\{ g \in \left( \begin{array}{ccc} \fO & \fO & \fO \\ \p^n & 1+ \p^n & \fO \\ \p^n & \p^n & \fO \end{array} \right) : \det (g) \in \fO^\times\right\} \cap U ,$$ which are the analogs of the Klingen subgroups introduced by Roberts and Schmidt \cite[Equation (2.5)]{RobertsSchmidt} in the $\mathrm{GSp(4)}$ setting. 
Lemma \ref{lem-AA:index} gives the decomposition of $\Gamma_n$ into $A_n$-cosets in terms of the group of trace zero elements of $\fO/\p^n$.  The sizes of these groups can then be computed using group cohomology and the integral normal basis theorem for unramified extensions of local fields (see Lemma \ref{lem:cardinality-E0}). In particular, Lemma \ref{lem-AA:index} reduces the computation of the {volume} of $\Gamma_n$ to the calculation of the indices $[\Gamma_0:A_n]$.
\item Consider the subgroups $B_n$ of $\Gamma_n$ given by $$B_n = \left\{ g \in \left( \begin{array}{ccc} \fO & \fO & \fO \\ \p^n & \fO & \fO \\ \p^n & \p^n & \fO \end{array} \right) : \det (g) \in \fO^\times\right\} \cap U,$$ and note that $B_n$ contains $A_n$. Lemma \ref{lem:AB-decomposition} provides the short exact sequence $$ \xymatrix{ 1 \ar[r] & A_n \ar[r] & B_n  \ar[r] & E_{\p^n}^1 \ar[r] & 1 },$$ 
 where $E_{\p^n}^1$ is the set of elements in $\fO$ with $x \overline{x}=1$, reduced mod $\p^n$.  The cardinality of $E_{\p^n}^1$ may be calculated using Hilbert's Theorem 90 (see Lemma \ref{lem:cardinality-E1}). 
\item Calculate the index of $B_n$ in $B_1$. This is the content of Lemma \ref{lem:volB} and uses the Iwahori factorization of reductive groups in order to state a recursive relation between the index of $B_n$ in $B_{n-1}$.
\item Finally, the index of $B_1$ in $\Gamma_0$ is computed in Lemma \ref{lem:indexB1} by appealing to the Bruhat decomposition of the reduction modulo $\p$ of $B_1$.
\end{enumerate}

Altogether these results compile into Theorem \ref{theorem}. Throughout the proofs we may assume $n \geqslant 1$, since we normalize $\Gamma_0$ to have measure one. 
\subsection{Comments on ramification}
Even though Miyauchi's theory of newforms has only been established for unitary groups attached to unramified extensions $E/F$, the subgroups $\Gamma_n$ can be defined without this assumption and it would be natural to want to extend the results of this paper to allow for ramification. We discuss here some modifications it would induce. 

It seems reasonable to expect that the proof of Theorem \ref{theorem} given in this paper could be adapted to allow for tame ramification in $E/F$. The generalization would force us to adapt most of the group cohomology results in Section 2 and introduce casework according to the parity of the residue characteristic. For example, since $G$ acts trivially on $E_\p$ when $E/F$ is ramified, we would have instead $H^1(G, E_{\p}^\times ) \simeq \Z/2\Z$ in Lemma \ref{multiplicativeHnlem}, except when the residue field has characteristic $2$. 

On the other hand, allowing $E/F$ to be wildly ramified seems to be more difficult. The hypothesis that $E/F$ is at most tamely ramified is used in Lemma \ref{tame} and also in the proof of Lemmas \ref{lem:volB} and \ref{lem:indexB1} to compute the cardinalities of certain unipotent subgroups. Underlying these is the existence of a normal $\o$-basis of $\fO$, which according to Noether's integral normal basis theorem happens if and only if $E/F$ is at most tamely ramified.

Since these generalizations would increase the length of the paper and it is not clear that the Miyauchi subgroups take the same shape for ramified extensions, we leave aside the question of ramification for the remainder of this work.

\subsection{Acknowledgments}

The authors are grateful to Jean Michel and Brooks Roberts for enlightening discussions, and to the Institut Henri Poincar\'e for providing workspace during several of the authors' research visits to Paris. We also thank the referee for valuable comments. The first author has been partially supported by the Deutscher Akademischer Austauschdienst (DAAD). The second author was supported by Swiss National Science Foundation grant PZ00P2\_168164.

\section{Group cohomology lemmas}
The following lemmas are standard results in group cohomology that will be used often in what follows. Given a group $G$ and a $G$-module $M$, we denote by $H^n(G,M)$ the $n$th cohomology group of $G$. See Serre \cite[Ch.VII]{SerreLF} for more definitions and background.

\begin{lem}\label{tame}
Let $A$ be a Dedekind domain, $K$ its field of fractions, $L$  an at most tamely ramified finite Galois extension of $K$, $G$ its Galois group, and $B$ the integral closure of $A$ in $L$. We have $H^n(G, B)= 0$ for all $n {\geqslant} 1$. 
\end{lem}
\begin{proof}  See e.g.\ \cite[(6.1.10) Theorem]{NWS}.
%By e.g. \cite[Th\'eor\`eme II.1]{Martinet}, $B$ is a projective $A[G]$-module, and thus is relatively projective. Since $G$ is a finite group, $B$ is relatively projective if and only if it is relatively injective \cite[page 233, Proposition 1.1]{CartanEilenberg}. Lastly, by  \cite[Chapter VII, Proposition 1]{SerreLF} we obtain $H^n(G,B) = 0$ for all $n\geqslant 1$. 
\end{proof}
\begin{cor}\label{cor:gcoh}
Let $E$ be an unramified quadratic extension of a non-archimedean local field $F$ with Galois group $G$ and ring of integers $\fO$. We have that $H^n(G, \fO) = 0$ for all $n\geqslant 1$. 
\end{cor}
Let $\varpi$ be a uniformizer of $\mathfrak{p}$. Since $E/F$ is unramified, $\varpi$ is also a uniformizer of $\fp$. Let $E_{\fp^n} = \fO/\fp^n$.
Corollary \ref{cor:gcoh} also descends to the finite additive rings $E_{\fp^n}$, as the next lemma shows. 
\begin{lem}\label{additiveHnlem}
We have that $H^{m}(G, E_{\fp^n}) = 0$ for all $m,n \geqslant 1$. 
\end{lem}
\begin{proof}
We have the short exact sequence $$\xymatrix{1 \ar[r] & \fO \ar[r]^-{\times\varpi^n} & \fO \ar[r] & E_{\p^n} \ar[r] & 1},$$ and taking the long exact sequence in cohomology we have for all $m, n \geqslant 1$ that $$\xymatrix{ \cdots \ar[r] & H^m(G,\fO) \ar[r] & H^m(G, E_{\p^n}) \ar[r] & H^{m+1}(G,\fO) \ar[r] & \cdots}.$$ Since the outer two terms vanish by Corollary  \ref{cor:gcoh}, we also have $H^m(G, E_{\p^n})=0$ for all $m, n\geqslant 1$. 
\end{proof}
We also consider finite multiplicative groups. Let $E_{\fp^n}^\times= (\fO/\fp^n)^\times$. 
\begin{lem}\label{multiplicativeHnlem}
We have $H^1(G,E_{\fp^n}^\times)=0$ for all $n\geqslant 1$. 
\end{lem}
\begin{proof}
The case $n =1$ is just Hilbert's Theorem 90: since $E/F$ is unramified, $G$ acts on the field $E_\fp$ by Galois automorphisms, and so $H^1(G, E_{\fp}^\times) =0$. 

Assume henceforth that $n \geqslant 2$. Hilbert's Theorem 90 no longer applies since $E_{\p^n}$ is not a field, but we shall still manage to show that $H^1(G,E_{\fp^n}^\times)=0$ as follows. By considering the sequence  $$\xymatrix{ 1 \ar[r] & 1+\varpi^n\fO \ar[r] & \fO^\times \ar[r] & E_{\p^n}^\times \ar[r] & 1},$$ and taking the long exact sequence in cohomology we get \begin{equation}\label{les1}\xymatrix{ \cdots \ar[r] & H^1(G,\fO^\times) \ar[r] & H^1(G, E_{\p^n}^\times) \ar[r] & H^2(G,1+\varpi^n\fO ) \ar[r] & \cdots}.\end{equation} 

We would like to show that $H^2(G,1+\varpi^n\fO )=0$. For any $m\geqslant 1$, consider the exact sequence $$\xymatrix{ 1 \ar[r] & 1+\varpi^{n+m}\fO \ar[r] & 1+\varpi^{n}\fO \ar[r] & \fp^n/\fp^{n+m} \ar[r] & 1}.$$ By the long exact sequence in cohomology, 
we see that \begin{equation} H^2(G, 1+\varpi^{n+m}\fO) \simeq H^2(G,1+\varpi^n\fO ),\end{equation} since the outer two cohomology groups in the long exact sequence vanish by Lemma \ref{additiveHnlem}.  For $m$ sufficiently large, we have $ 1+\varpi^{n+m}\fO \simeq \fO$ as $G$-modules, so that $H^2(G,1+\varpi^{n+m}\fO)$ and thus $H^2(G, 1+\varpi^{n}\fO)$ vanishes by Corollary \ref{cor:gcoh}. 

Returning to \eqref{les1}, to prove the lemma, it now suffices to show that  $H^1(G, \fO^\times)$ vanishes. Consider the sequence $$\xymatrix{ 1 \ar[r] & \fO^\times \ar[r] & E^\times \ar[r]^-{\rm val} & \Z \ar[r] & 1 }.$$ Taking the long exact sequence in cohomology and by Hilbert's Theorem 90 we have $$\xymatrix{1 \ar[r] & \o^\times \ar[r] & F^\times \ar[r]^-{\rm val} & \Z \ar[r] & H^1(G,\fO^\times) \ar[r] & 1}.$$ The valuation map here is surjective since $E/F$ is unramified. Therefore we have $H^1(G,\fO^\times)=0$, and thus $H^1(G, E_{\p^n}^\times)=0$ from \eqref{les1}.
\end{proof}

Let $E_{\fp^n}^0$ be the subgroup of $E_{\fp^n}$ of trace-zero elements, that is $E_{\fp^n}^0 = \{x \in E_{\fp^n}: x+\overline{x} = 0\}$. Likewise, write $E_{\fp^n}^1$ for the subgroup of $E_{\p^n}$ consisting of norm-one elements, i.e. $E_{\fp^n}^1=\{x \in E_{\fp^n}^\times : x \overline{x}=1\}$.  Our intended application of the above group cohomology lemmas is to calculate the cardinalities of the finite groups $E_{\fp^n}^0$ and $E_{\fp^n}^1$. For any $G$-module $M$, we write $Z^1(G,M)$ for the group of $1$-cocycles and $B^1(G,M)$ for the set of $1$-coboundaries.  
\begin{lem}
\label{lem:cardinality-E1}
For every $n \geqslant 1$, 
\begin{equation} |E_{\p^n}^1| = q_F^{n-1}(q_F+1).\end{equation}
\end{lem}
\begin{proof}
Let $\sigma$ denote the non-identity element of $G$.  We have \begin{equation}\label{isom3} Z^1(G, E_{\p^n}^\times) \simeq E_{\p^n}^1,\end{equation} where the map is given by $\xi \mapsto \xi(\sigma)$. 
We have an exact sequence $$\xymatrix{ 1 \ar[r] & \left( \frac{\o}{\p^n \cap \o}\right)^{\times} \ar[r] & E_{\p^n}^\times \ar[r]^-f & Z^1(G,E_{\p^n}^\times)},$$ where $f$ is given by the composition of $x \mapsto x/\overline{x}$ and \eqref{isom3}. The image of $f$ is by definition $B^1(G,E_{\p^n}^\times)$. But, by Lemma \ref{multiplicativeHnlem}, we have $B^1(G,E_{\p^n}^\times)=Z^1(G,E_{\p^n}^\times)$. Thus \begin{equation}|E_{\p^n}^1| = \frac{|E_{\p^n}^\times|}{|(\o/\p^n \cap \o)^\times|} = \frac{q_E^{n-1}(q_E-1)}{q_F^{n-1}(q_F-1)},\end{equation} which equals $q_F^{n-1}(q_F+1)$ since $q_E = q_F^2$, because $E/F$ is unramified.
\end{proof}
\begin{lem}
\label{lem:cardinality-E0}
For every $n \geqslant 1$, 
\begin{equation}|E_{\p^n}^0| = q_F^{n}.\end{equation}
\end{lem}
\begin{proof}
We have \begin{equation}\label{isom6} Z^1(G, E_{\p^n}) \simeq E_{\p^n}^0,\end{equation} where the map is given by $\xi \mapsto \xi(\sigma)$. Now consider the exact sequence $$\xymatrix{1 \ar[r] & \frac{\o}{\p^n \cap \o} \ar[r] & E_{\p^n} \ar[r]^-f & Z^1(G,E_{\mathfrak{p}^n})},$$ where the map $f$ is given by the composition of $u \mapsto u - \overline{u}$ and \eqref{isom6}. The image of $f$ is $f(E_{\p^n}) = B^1(G,E_{\p^n}) = Z^1(G, E_{\fp^n})$ by Lemma \ref{additiveHnlem}. Therefore \begin{equation}|E_{\p^n}^0| = \frac{|E_{\p^n}|}{|\frac{\o}{\p^n \cap \o}|} \cdot \big|H^1(G,E_{\p^n})\big|.\end{equation}
Since $E/F$ is unramified, we have $q_E=q_F^2$. We have $|E_{\p^n}| = q_E^n$ and $|\frac{\o}{\p^n \cap \o}| = q_F^n$ (again, by the unramified hypothesis). Therefore, we have shown that $$|E_{\p^n}^0| = q_E^{n/2} = q_F^{n}.$$
\end{proof}

\section{From $A_n$ to $\Gamma_n$}
Recall that we have chosen a common uniformizer $\varpi$ of $\mathfrak{p}$ and $\fp$. Introduce, for $n \geqslant 0$ and $u \in \fO$, 
\begin{equation} t_n(u) =  \left( \begin{array}{ccc} 1 & & u\varpi^{-n} \\ & 1 & \\  & & 1 \end{array}\right) \end{equation} 
\noindent which is in $\Gamma_n$ if and only if $u+\overline{u}=0$, and 
\begin{equation} \sigma_n  = \left( \begin{array}{ccc}  & & \varpi^{-n} \\ & 1 & \\  \varpi^n & &  \end{array}\right) \in \Gamma_n .\end{equation}

For every $n \geqslant 0$ and $k \geqslant 0$, let \begin{equation}C_n^{(k)} = \left\{ g \in \left( \begin{array}{ccc} \fO & \fO & \fp^{k-n} \\ \p^n & 1+ \p^n & \fO \\ \p^n & \p^n & \fO \end{array} \right) : \det (g) \in \fO^\times\right\} \cap U.\end{equation}

\begin{lem}
\label{lem-A:index}
For $E/F$ unramified we have
\begin{equation}
\left[C_n^{(k)} : C_n^{(k+1)}\right] = \begin{cases} q_F + 1 & \text{ if } \ k=0 \\ q_F & \text{ if } \ 1 \leqslant k <n.\end{cases}
\end{equation}
\end{lem}

\begin{proof}
Introduce
\begin{equation}\label{greekmatrix}
g = \left( \begin{array}{ccc} \alpha & \beta & \gamma \varpi^{k-n} \\ \delta \varpi^n & 1+ \epsilon \varpi^n & \zeta \\ \eta \varpi^n & \theta \varpi^n & \iota \end{array} \right),
\end{equation}
\noindent a typical element in $C_n^{(k)}$, where every Greek letter is a generic element in $\fO$ except the uniformizer $\varpi$.  By calculating the determinant of the matrix \eqref{greekmatrix}, we find that $\det(g) \equiv \alpha \iota \mod {\fp^k}$. Since $\det(g) \in \fO^\times$, if $k \geqslant 1 $, then {we have} \begin{equation}\label{aliointertable}\alpha ,\iota \in \fO^\times.\end{equation}

If $u \in \fO$ satisfies $u+\overline{u} = 0$, then $t_{n-k}(u) \in C^{(k)}_n$. For such $u$ and any $0\leqslant k <n$ we claim that
\begin{equation}\label{s3claim} t_{n-k}(u) C_{n}^{(k+1)} = \left\{g \in C_n^{(k)}: \iota \in \fO^\times, u-\iota^{-1} \gamma \in \fp\right\},\end{equation}
where $\iota$ and $\gamma$ on the right side of \eqref{s3claim} refer to entries of $g \in C_n^{(k)}$ as in \eqref{greekmatrix}. 

We first show $\subseteq$. Suppose $g \in C_{n}^{(k+1)}$ is written as in \eqref{greekmatrix}, but with $k+1$ in the upper right entry in place of $k$. Then $\iota \in \fO^\times$ by \eqref{aliointertable} for any $k\geqslant 0$, since $k+1\geqslant 1$.  Calculating, we have $$  t_{n-k}(u)g = \left( \begin{array}{ccc} \alpha + u \eta \varpi^k & \beta + u \theta \varpi^k & \gamma \varpi^{k+1-n} + u \iota \varpi^{k-n} \\ * & * & * \\ * & * & \iota \end{array} \right).$$ To show that $t_{n-k}(u)g$ is in the right hand side of \eqref{s3claim}, it then suffices to show that $\iota \in \fO^\times$ and $u- \iota^{-1}(\gamma \varpi+ u \iota) \in \fp$. The first of these is {\eqref{aliointertable}}, and the second is clear after expanding out. 

Now we show $\supseteq$. Suppose $g \in C^{(k)}_n$ is written as \eqref{greekmatrix}, $\iota \in \fO^\times$, and $u \in \fO$ is such that $u+\overline{u}=0$ and $u - \iota^{-1} \gamma \in \fp$. Then 
\begin{equation*}
t_{n-k}(-u) g = \left( \begin{array}{ccc} \alpha & \beta & (\gamma - \iota u) \varpi^{k-n} \\ *& *&* \\ * & * & * \end{array} \right),
\end{equation*}
but $\gamma - \iota u  \in \fp$ by hypothesis, so that $t_{n-k}(-u) g \in C_n^{(k+1)}$. The claim \eqref{s3claim} follows. 

Note that if $k\geqslant 1$, then the hypothesis $\iota \in \fO^\times$ can be omitted from the right hand side of \eqref{s3claim} because it follows automatically from \eqref{aliointertable}. Therefore under the hypothesis that $k\geqslant 1$ we see that \begin{equation}C_n^{(k)} = \bigsqcup_{u \in E_{\fp}^0} t_{n-k}(u) C_{n}^{(k+1)}.\end{equation} The result now follows from Lemma \ref{lem:cardinality-E0}.

It remains to treat the case $k=0$. Under this hypothesis we claim that
\begin{equation}\label{s3claim2}
\sigma_n C^{(1)}_n = \left\{ g \in C_n^{(0)}: \iota \in \fp\right\}.
\end{equation}
To see $\subseteq$, write $g \in C^{(1)}_n$ in the form \eqref{greekmatrix}, and calculate 
$$ \sigma_n g = \left( \begin{array}{ccc} \eta & \theta  & \iota \varpi^{-n}  \\ \delta \varpi^n & 1+\epsilon \varpi^n & \zeta \\ \alpha \varpi^n & \beta \varpi^n & \gamma \varpi \end{array} \right).$$
To see $\supseteq$, write $g \in C^{(0)}_n$ in the form \eqref{greekmatrix}, and recall the hypothesis $\iota \in \fp$. Note that $\sigma_n^2=1$, and calculate
 $$ \sigma_n g = \left( \begin{array}{ccc} \eta & \theta  & \iota \varpi^{-n}  \\ \delta \varpi^n & 1+\epsilon \varpi^n & \zeta \\ \alpha \varpi^n & \beta \varpi^n & \gamma  \end{array} \right),$$
 from which we see that $\iota \varpi^{-n} \in \fp^{1-n}$, as desired. 
 
 Combining \eqref{s3claim} and \eqref{s3claim2}, we see that 
 \begin{equation}C_n^{(0)} =\sigma_n C^{(1)}_n \sqcup \bigsqcup_{u \in E_{\fp}^0} t_{n}(u) C_{n}^{(1)},\end{equation} from which the result again follows by Lemma \ref{lem:cardinality-E0}.
\end{proof}

Since $\Gamma_n = C_n^{(0)}$ and $A_n = C_n^{(n)}$, the index of $A_n$ in $\Gamma_n$ follows by induction.

\begin{lem}
\label{lem-AA:index}
For every $n\geqslant 1$,
\begin{equation}
[\Gamma_n: A_n ] = q_F^{n-1}(q_F + 1).
\end{equation}
\end{lem}

\section{From $A_n$ to $B_n$}
This section aims at reducing the study of the indices of $A_n$ in $\Gamma_n$ to those of the more paramodularly-shaped $B_n$. 
Consider the homomorphism $$f:B_n \to E_{\p^n}^\times $$ given by $$\left(\begin{array}{ccc} \alpha & \beta & \gamma \\ \delta \varpi^n & \epsilon & \zeta \\ \eta \varpi^n & \theta \varpi^n & \iota \end{array}\right) \longmapsto \epsilon \mod {\p^n}.$$ 
\begin{lem}
\label{lem:AB-decomposition}
If $n\geqslant 1$, the sequence \begin{equation}\xymatrix{ 1 \ar[r] & A_n \ar[r] & B_n \ar[r]^-f & E_{\p^n}^1 \ar[r] & 1}\end{equation} is short exact. In particular, \begin{equation}[\Gamma_0:A_n] = [\Gamma_0:B_n] \cdot |E_{\p^n}^1|.\end{equation} 
\end{lem}
\begin{proof}
Clearly, the kernel of $f$ is exactly $A_n$. It suffices to show that $f$ has image $E_{\p^n}^1$. Let $$b = \left(\begin{array}{ccc} \alpha & \beta & \gamma \\ \delta \varpi^n & \epsilon & \zeta \\ \eta \varpi^n & \theta \varpi^n & \iota \end{array}\right)$$ be an arbitrary element of $B_n$, where Greek letters are elements of $\fO$. Then since $${}^t \overline{b} J b = J,$$ we have \begin{equation}(\beta \overline{\theta} + \overline{\beta}\theta)\varpi^n + \epsilon\overline{\epsilon} = 1,\end{equation} by expanding the center entry of $ {}^t \overline{b} J b = J,$ and since $\varpi$ was chosen so that $\varpi \in \mathfrak{p} \subset F$.   Therefore we have $\epsilon \bar{\epsilon} \equiv 1 \mod{\p^n}$, so that $f(B_n) \subseteq E^1_{\p^n}$. On the other hand, we have $$\left(\begin{array}{ccc} 1 & & \\ & \epsilon & \\ & & 1\end{array}\right) \in B_n$$ for any $\epsilon \in \fO$ such that $\epsilon \overline \epsilon = 1$ by a direct calculation. Therefore $f(B_n) \supseteq E^1_{\p^n}$, and the lemma is proved.
\end{proof}

\section{Index of $B_n$}
To compute the index of $B_n$ in $\Gamma_0$ we will need the Iwahori decomposition, our reference for which is \cite{macdonald_spherical_1971}.  The index of $B_n$ in $\Gamma_0$ is given in the following lemma. 
\begin{lem}
\label{lem:volB}
 For every $n \geqslant 1$ we have
\begin{equation}
[{\Gamma_0:B_n}] = {q_F}^{3(n-1)} [{\Gamma_0:B_1}].
\end{equation}
\end{lem}
\proof We use the Iwahori decomposition for the group $U$, which may be quickly derived from that of $\GL_3 (E)$. In \cite[Proposition 2.6.4]{macdonald_spherical_1971}, taking $G= \GL_3(E)$ and $S$ to be the convex hull of the following 4 points $$\{(0,0,0), (n,0,0),(0,0,-n),(n,n,0)\},$$ we obtain the factorization 
\begin{equation}
\label{GL3-iwahori}
\left( \begin{array}{ccc} \fO &\fO&\fO \\ \p^n & \fO & \fO \\ \p^n & \p^n & \fO \end{array}\right) = 
\left( \begin{array}{ccc} 1 & &  \\ \p^n & 1 & \\ \p^n & \p^n & 1 \end{array}\right)
\left( \begin{array}{ccc} \fO^\times & &  \\ & \fO^\times & \\  & & \fO^\times \end{array}\right)
\left( \begin{array}{ccc} 1 & \fO & \fO\\ & 1 & \fO \\  & & 1 \end{array}\right),
\end{equation}
where each matrix represents the indicated subgroup of $\GL_3(E)$. We may now realize the group $B_n$ as the fixed points of the left hand side of \eqref{GL3-iwahori} by the involution $g \mapsto J {}^t \overline{g}^{-1} J$. Let $g = n_{-} t n_+$ be the factorization of a typical element $g$ as in \eqref{GL3-iwahori}. Since $J {}^t \overline{g}^{-1} J = J {}^t \overline{n}_{-}^{-1} J J {}^t \overline{t}^{-1} J J {}^t \overline{n}_{+}^{-1} J,$ we derive for any $n\geqslant 1$ that
\begin{equation}
\label{B-iwahori}
B_n = 
\left( \begin{array}{ccc} 1 & &  \\ \p^n & 1 & \\ \p^n & \p^n & 1 \end{array}\right)
\left( \begin{array}{ccc} \fO^\times & &  \\ & \fO^\times & \\  & & \fO^\times \end{array}\right)
\left( \begin{array}{ccc} 1 & \fO & \fO\\ & 1 & \fO \\  & & 1 \end{array}\right),
\end{equation}

\noindent where each matrix appearing in \eqref{B-iwahori} represents the intersection of the indicated subgroup with $U$. Moreover, the matrices of the first subgroup appearing in \eqref{B-iwahori} are exactly those of the form %\didier{Here is the critical point we need to change}
\begin{equation}
\label{55}
\left(
\begin{array}{ccc}
1 & & \\
\alpha \varpi^n & 1 & \\
\beta \varpi^n & - \bar{\alpha} \varpi^n & 1
\end{array}
\right)
\quad \text{where} \quad
\beta + \bar\beta + \alpha\bar\alpha \varpi^n = 0.
\end{equation}

\noindent For $n\geqslant 1$ this leads to the decomposition 
\begin{equation}
\left( \begin{array}{ccc} 1 & &  \\ \p^{n-1} & 1 & \\ \p^{n-1} & \p^{n-1} & 1 \end{array}\right) = \bigsqcup_{\substack{\alpha, \beta \in E_\p \\ \beta + \bar\beta + \alpha\bar\alpha \varpi^n = 0 }} 
\left( \begin{array}{ccc} 1 & &  \\ \alpha \varpi^{n-1} & 1 & \\ \beta \varpi^{n-1} & - \bar\alpha \varpi^{n-1} & 1 \end{array}\right)\left( \begin{array}{ccc} 1 & &  \\ \p^{n} & 1 & \\ \p^{n} & \p^{n} & 1 \end{array}\right),
\end{equation}
where the condition $\beta + \bar\beta + \alpha\bar\alpha \varpi^n = 0$ on $\alpha, \beta \in E_\p$ is understood to mean that we sum over those cosets for which there exists a representative in $\fO$ satisfying the indicated equation.

Given $\alpha, \beta$ as above, it is clear by reducing the condition modulo $\p$ that $\beta \in E_\p^0$. Conversely, given $\alpha^* \in E_\p$ and $\beta^* \in E_\p^0$, there exists a representative $\beta \in \fO$ for $\beta^*$ such that $\beta + \bar\beta=0$. Let $\alpha$ be any choice of representative for $\alpha^*$. Then $\alpha\bar\alpha \varpi^n \in \mathfrak{p}$. Since $E/F$ is at most tamely ramified we have as a consequence of the integral normal basis theorem (see e.g.\ \cite[\S5 Theorem 2]{Froehlich}) that $\mathrm{Tr}\, \p = \mathfrak{p}$, and so there exists $z \in \p$ such that $\alpha \bar\alpha\varpi^n = z + \bar z$. Then $\beta - z$ is another representative for $\beta^* \in E_\p^0$, and $(\beta-z) + \overline{(\beta-z)} + \alpha\bar\alpha \varpi^n = 0$. Thus, the decomposition \eqref{B-iwahori} may be re-written as 
%Note that since $E/F$ is at most tamely ramified we have by e.g.\ \cite[\S5 Theorem 2]{Froehlich} that $\mathrm{Tr}\, \p = \mathfrak{p}$. We have $\alpha \bar\alpha \varpi^n \in \mathfrak{p}$ since $n \geqslant 1$, so for any $\alpha, \beta  \in \fO/\p$ with representatives satisfying $\beta + \bar\beta + \alpha\bar\alpha \varpi^n = 0$ there always exists a representative $\beta\in E_\p^0$ satisfying $\beta + \bar\beta + \alpha \bar\alpha \varpi^n =0$. From this and the decomposition \eqref{B-iwahori} we find \sout{ from which, with the decomposition \eqref{B-iwahori}, we derive}}
\begin{equation}
\label{B-recursive-relation}
B_{n} = 
\bigsqcup_{{\substack{\alpha \in E_\p \\ \beta \in E_\p^0}}} 
\left( \begin{array}{ccc} 1 & &  \\ \alpha {\varpi^{n}} & 1 & \\ \beta {\varpi^{n}} & - \bar\alpha {\varpi^{n}} & 1 \end{array}\right)
B_{n+1}.
\end{equation}

\noindent We have $|E_\p|=q_E =q_F^2$ and $|E_\p^0|=q_F$ by Lemma \ref{lem:cardinality-E0}, so this finishes the proof by induction. \qed

\section{Index of $B_1$}

We use reduction modulo $\p$ and the Bruhat decomposition, and recall now some useful lemmas to do so.

\begin{lem}
\label{lem:conservation}
Let $X, G$ be two groups, $H$ a subgroup of $G$ and $\phi : X \to G$ a surjective morphism. Then
\begin{equation}
[G:H] = [\phi^{-1}(G) : \phi^{-1}(H)].
\end{equation}
\end{lem}

\proof We write $G/H$ for the set of left cosets of $H$ in $G$, that is we have 
$$ G = \bigsqcup_{g \in G/H} gH \quad \text{ and so } \quad \phi^{-1}(G) = \bigsqcup_{g \in G/H} \phi^{-1}(gH).$$
Since $\phi$ is surjective, for any representative of $g \in G/H$ we may choose a preimage $g' \in \phi^{-1}(g)$. Then $\phi^{-1}(gH) = g'\phi^{-1}(H)$ and the lemma follows.
\qed

\begin{lem}[Hensel's Lemma]
\label{lem:surjectivity}
Let $\mathbf{G}$ be a smooth scheme over the ring of integers $\o$ of a non-archimedean local field. Let $\mathfrak{p}$ the corresponding maximal ideal. Then the reduction modulo $\mathfrak{p}$
\begin{equation}
\mathbf{G}(\o) \longrightarrow \mathbf{G}( F_\mathfrak{p}) 
\end{equation}

\noindent is surjective.
\end{lem}

\proof This is a standard version of Hensel's lemma in algebraic geometry. For precise details, see \cite[Th\'eor\`eme 18.5.17]{EGA4pt4}.
\qed

\begin{lem}
\label{lem:indexB1}
We  have
\begin{equation}
[\Gamma_0 : B_1] = q_F^3+1.
\end{equation}
\end{lem}

\proof 
Consider the unitary group $\mathbf{U}$ defined over $\o$ and its standard Borel subgroup $\mathbf{B}$. We have $\mathbf{U}(\o) = \Gamma_0$ and $\mathbf{B}(\o)= B_1$. Within this proof, we write $U= \mathbf{U}(F_{\mathfrak{p}})$ and $B= \mathbf{B}( F_{\mathfrak{p}})$. By Hensel's Lemma (Lemma \ref{lem:surjectivity}) we have surjective morphisms
$$\Gamma_0 \twoheadrightarrow U \quad \text{ and } \quad B_1 \twoheadrightarrow B.$$ By Lemma \ref{lem:conservation} we then have that $[\Gamma_0: B_1] = [U : B]$, so that the problem reduces to finite fields. 

We apply the Bruhat decomposition for $U$, following \cite[Chapter 4]{digne_representations_2020}. Consider the group $\GL_3$ over $ \overline{F}_{\mathfrak{p}}$ and the involution $\tau$ defined by $g \mapsto J  {}^t \overline{g}^{-1} J$, where $\overline{g}$ denotes the Frobenius automorphism applied to $g$, i.e.\ the entries of $g$ raised to the $p$th power. The fixed points of $\GL_3$ by this involution are exactly the group $U$ (cf.\ \cite[Example 4.3.3]{digne_representations_2020}).  Let $T$ be the standard maximal torus of $\GL_3$, $N=N(T)$ its normalizer, and $W=W(T)$ the associated Weyl group. Then, we have a Bruhat decomposition of $U$ of the form 
\begin{equation}
U = \bigsqcup_{w \in W^\tau} B w U_w^{\tau},
\end{equation}
with $B w U_w^{\tau}$ being a direct product \cite[Lemma 3.2.7]{digne_representations_2020}, where $$U_w= \prod_{\{\alpha \in \Phi^+: w(\alpha)<0\}} U_\alpha$$ and $U_\alpha$ is the root space associated to the positive root $\alpha \in \Phi^+$ (see \cite[Proposition 4.4.1(ii)]{digne_representations_2020}). Warning: $U_w$ and $U_\alpha$, which are subgroups of $\GL_3$, should not be confused with the unitary group $U$.  

We proceed to calculate $W^\tau$ and $U_w^\tau$. By \cite[Proposition 4.4.1(i)]{digne_representations_2020} we have that $W^\tau = N(T)^\tau/T^\tau$. We can calculate explicitly that $$T^\tau = \left\{\left( \begin{array}{ccc} a & & \\ & b & \\ & & c\end{array}\right) : \ a\overline{c} = 1, b\overline{b}=1\right\}$$  and  $$N(T)^\tau = T^\tau \sqcup \left\{\left( \begin{array}{ccc} & & c \\ & b & \\ a & & \end{array}\right) : \ a\overline{c} = 1, b\overline{b}=1\right\}.$$
Thus $W^\tau$ consists of two elements, which may be represented by $I$ and $J$ (see \eqref{Jdef}). 

Now we calculate $U_w^\tau$ for $w=I$ and $w=J$. We first determine the sets $\{\alpha \in \Phi^+: w(\alpha)<0\}$. Let $U_{ij}$, $1 \leqslant i\neq j \leqslant 3$ denote the root space of $\GL_3$ that is non-zero in the $ij$th entry. Let $\alpha_{ij}$ denote the corresponding root (see \cite[Theorem 2.3.1(i)]{digne_representations_2020}). 
The set $\Phi^+$ of positive roots corresponding to our choice of the standard upper triangular Borel is $\Phi^+ = \{\alpha_{12},\alpha_{13},\alpha_{23}\}$. The Weyl group $W$ acts on $T$ by conjugation, and thus on $\Phi$. Clearly, $I$ fixes $\Phi$, and one may compute that $J(\alpha_{12})= \alpha_{32}$, $J(\alpha_{13})= \alpha_{31}$, and $J(\alpha_{23}) = \alpha_{21}$. Thus, $\{\alpha \in \Phi^+: I(\alpha)<0\} = \varnothing$ and $\{\alpha \in \Phi^+: J(\alpha)<0\} = \Phi^+$. Therefore we have 
\begin{equation}
U = B \sqcup B J \left( \begin{array}{ccc} 1 & \star & \star \\ & 1 & \star \\ & & 1\end{array}\right)^\tau.
\end{equation}
This last group of unipotent matrices is in bijection with the set $\cup_{\alpha \in \mathfrak{O}/\mathfrak{P}} N_\alpha,$ where $N_\alpha = \{\beta \in \mathfrak{O}/\mathfrak{P} : \beta + \bar\beta + \alpha \bar\alpha = 0\}$ (cf.\ \eqref{55}). Recall (e.g.\ \cite[\S5 Theorem 2]{Froehlich}) that $\mathrm{Tr}\, \fO = \o$, since $E/F$ is at most tamely ramified. Thus, there exists $\theta \in \mathfrak{O}/\mathfrak{P}$ such that $\theta + \overline{\theta} = -\alpha \bar\alpha$. We then find that $N_\alpha = \theta + E_\p^0$, which is of cardinality $q_F$ by Lemma \ref{lem:cardinality-E0}. Thus, the cardinality of the group of unipotent matrices in question is then $q_Eq_F = q_F^3$.
We conclude that $[U:B] = q_F^3+1$, hence the Lemma follows. 
 \qed

We now collect the previous results. By Lemmas \ref{lem-AA:index} and \ref{lem:AB-decomposition} we have \begin{equation}\vol (\Gamma_n )= \frac{[\Gamma_n : A_n]}{[\Gamma_0 : A_n]} \vol (\Gamma_0)= \frac{q_F^{n-1} (q_F+1) }{[ \Gamma_0:B_n] | E_{\p^n}^1|},\end{equation} which by Lemmas \ref{lem:cardinality-E1} and \ref{lem:volB}is $= q_F^{3-3n}[\Gamma_0:B_1]^{-1}.$ Finally, by Lemma \ref{lem:indexB1} we conclude Theorem \ref{theorem}. 

%\bibliographystyle{alpha}
%\bibliography{myrefs}	

\begin{thebibliography}{NSW08}

\bibitem[Bar97]{Baruch}
E.~M. Baruch.
\newblock On the gamma factors attached to representations of {U(2,1)} over a
  $p$-adic field.
\newblock {\em Israel J. Math.}, 102:317--345, 1997.

\bibitem[BM18]{BrumleyMilicevic}
F.~Brumley and D.~Mili\'cevi\'c.
\newblock Counting cusp forms by analytic conductor.
\newblock arXiv:1805.00633, 2018.

\bibitem[Cas73]{casselman_results_1973}
W.~Casselman.
\newblock On some results of {A}tkin and {L}ehner.
\newblock {\em Math. Ann.}, 201:301--314, 1973.

\bibitem[Del73]{DeligneNewforms}
P.~Deligne.
\newblock Formes modulaires et repr\'{e}sentations de {${\rm GL}(2)$}.
\newblock In {\em Modular functions of one variable, {II} ({P}roc. {I}nternat.
  {S}ummer {S}chool, {U}niv. {A}ntwerp, {A}ntwerp, 1972)}, pages 55--105.
  Lecture Notes in Math., Vol. 349, 1973.

\bibitem[DM20]{digne_representations_2020}
F.~Digne and J.~Michel.
\newblock {\em Representations of finite groups of {L}ie type}, volume~95 of
  {\em London Mathematical Society Student Texts}.
\newblock Cambridge University Press, Cambridge, 2nd edition, 2020.

\bibitem[Fr67]{Froehlich}
A.~Fr\"{o}hlich.
\newblock Local fields.
\newblock In {\em Algebraic {N}umber {T}heory ({P}roc. {I}nstructional {C}onf.,
  {B}righton, 1965)}, pages 1--41. Thompson, Washington, D.C., 1967.

\bibitem[GPS84]{gelbart_automorphic_1983}
S.~Gelbart and I.~Piatetski-Shapiro.
\newblock Automorphic forms and {$L$}-functions for the unitary group.
\newblock In {\em Lie group representations, {II} ({C}ollege {P}ark, {M}d.,
  1982/1983)}, volume 1041 of {\em Lecture Notes in Math.}, pages 141--184.
  Springer, Berlin, 1984.

\bibitem[Gro67]{EGA4pt4}
A.~Grothendieck.
\newblock \'{E}l\'{e}ments de g\'{e}om\'{e}trie alg\'{e}brique. {IV}. \'{E}tude
  locale des sch\'{e}mas et des morphismes de sch\'{e}mas {IV}.
\newblock {\em Inst. Hautes \'{E}tudes Sci. Publ. Math.}, (32):361, 1967.

\bibitem[ILS00]{iwaniec_low_2000}
H.~Iwaniec, W.~Luo, and P.~Sarnak.
\newblock Low lying zeros of families of {$L$}-functions.
\newblock {\em Inst. Hautes \'{E}tudes Sci. Publ. Math.}, (91):55--131 (2001),
  2000.

\bibitem[JPSS81]{jacquet_conducteur_1981}
H.~Jacquet, I.~I. Piatetski-Shapiro, and J.~Shalika.
\newblock Conducteur des repr\'{e}sentations du groupe lin\'{e}aire.
\newblock {\em Math. Ann.}, 256(2):199--214, 1981.

\bibitem[KL06]{knightly_relative_2006}
A.~Knightly and C.~Li.
\newblock A relative trace formula proof of the {P}etersson trace formula.
\newblock {\em Acta Arith.}, 122(3):297--313, 2006.

\bibitem[Les18]{lesesvre_arithmetic_2018}
D.~Lesesvre.
\newblock {\em Arithmetic {Statistics} for {Quaternion} {Algebras}}.
\newblock PhD thesis, Universit\'{e} Paris 13, Paris, 2018.

\bibitem[Les20]{Lesesvre}
D.~Lesesvre.
\newblock Counting and equidistribution for quaternion algebras.
\newblock {\em Math. Z.}, 295:129--159, 2020.

\bibitem[Mac71]{macdonald_spherical_1971}
I.~G. Macdonald.
\newblock {\em Spherical functions on a group of {$p$}-adic type}.
\newblock Ramanujan Institute, Centre for Advanced Study in
  Mathematics,University of Madras, Madras, 1971.
\newblock Publications of the Ramanujan Institute, No. 2.

\bibitem[Miy06]{miyake_modular_1989}
T.~Miyake.
\newblock {\em Modular forms}.
\newblock Springer Monographs in Mathematics. Springer-Verlag, Berlin, english
  edition, 2006.
\newblock Translated from the 1976 Japanese original by Yoshitaka Maeda.

\bibitem[Miy13a]{MiyauchiConductors2}
M.~Miyauchi.
\newblock On epsilon factors attached to supercuspidal representations of
  unramified {${\rm U}(2,1)$}.
\newblock {\em Trans. Amer. Math. Soc.}, 365(6):3355--3372, 2013.

\bibitem[Miy13b]{Miyauchi3}
M.~Miyauchi.
\newblock On local newforms for unramified {${\rm U}(2,1)$}.
\newblock {\em Manuscripta Math.}, 141(1-2):149--169, 2013.

\bibitem[Miy18]{Miyauchi4}
M.~Miyauchi.
\newblock On {$L$}-factors attached to generic representations of unramified
  {$\rm U(2,1)$}.
\newblock {\em Math. Z.}, 289(3-4):1381--1408, 2018.

\bibitem[NSW08]{NWS}
J\"urgen Neukirch, Alexander Schmidt, and Kay Wingberg.
\newblock {\em Cohomology of number fields}, volume 323 of {\em Grundlehren der
  Mathematischen Wissenschaften [Fundamental Principles of Mathematical
  Sciences]}.
\newblock Springer-Verlag, Berlin, second edition, 2008.

\bibitem[Rog90]{Rogawski}
J.~D. Rogawski.
\newblock {\em Automorphic representations of unitary groups in three
  variables}, volume 123 of {\em Annals of Mathematics Studies}.
\newblock Princeton University Press, Princeton, NJ, 1990.

\bibitem[RS07]{RobertsSchmidt}
B.~Roberts and R.~Schmidt.
\newblock {\em Local newforms for {GS}p(4)}, volume 1918 of {\em Lecture Notes
  in Mathematics}.
\newblock Springer, Berlin, 2007.

\bibitem[Ser79]{SerreLF}
J.-P. Serre.
\newblock {\em Local fields}, volume~67 of {\em Graduate Texts in Mathematics}.
\newblock Springer-Verlag, New York-Berlin, 1979.
\newblock Translated from the French by Marvin Jay Greenberg.

\end{thebibliography}

\end{document}